\documentclass[leqno,12pt]{amsart} 
\title[$G_2$ as twisted ring group]{The compact  exceptional Lie algebra $\mathfrak g^c_2$ \\ as a twisted ring group}
\author{Cristina Draper}
\usepackage{amssymb,color, bm}

\thanks{${}^*$ Supported   by the Spanish Ministerio de Ciencia e Innovaci\'on   through projects  PID2019-104236GB-I00 and PID2020-118452GB-I00, all of them with FEDER funds, and by the group FQM-336 by Junta de Andaluc\'{\i}a} 

\usepackage{amssymb, mathrsfs, amsmath, tabu, amsthm}
\usepackage{hyperref}
\usepackage{tikz-cd, graphicx} 
\usepackage{pdfpages}
\usepackage{enumitem}
\usepackage{ragged2e}
\usetikzlibrary{decorations.markings}

 \newcommand{\la}{\mathfrak{L}} 
\newcommand{\f}[1]{\mathfrak{#1}}
\newcommand{\bb}[1]{\mathbb{#1}}

\newcommand{\q}{\theta P}  
 \newcommand{\ba}[3]{\begin{pmatrix}E_{#1}^{#2,#3}\\F_{#1}^{#2}\end{pmatrix}}   
  
  \DeclareMathOperator{\Der}{Der}
\DeclareMathOperator{\ad}{ad}

\theoremstyle{plain} 
\newtheorem{theorem}{Theorem} 
\newtheorem{lemma}[theorem]{Lemma}
\newtheorem{coro}[theorem]{Corollary}

\theoremstyle{definition} 
\newtheorem{definition}[theorem]{Definition}

\newtheorem{notation}[theorem]{Notation}

\begin{document}


\maketitle 
\begin{abstract} 
A new highly symmetrical model of the compact Lie algebra $\f{g}^c_2$ is provided as a twisted ring group for the group  $\mathbb{Z}_2^3$ and the ring $\mathbb{R}\oplus\mathbb{R}$.
The model is self-contained and can be used without previous knowledge on roots, derivations on octonions or cross products. In particular, it provides an orthogonal basis with integer structure constants, consisting entirely of semisimple elements, which is a generalization of the Pauli matrices in $\mathfrak{su}(2)$ 
and of the Gell-Mann matrices in $\mathfrak{su}(3)$. As a bonus, the split  Lie algebra $\mathfrak{g}^*_2$ is also seen as a twisted ring group.
\end{abstract}

	
 \section{Introduction and main results }  
 
 Since the discovery of exceptional Lie algebras at the end of the 19th century, the scientific community  
 has not stopped trying to understand them more deeply. And surprisingly, it has not stopped succeeding.
 The case of $\f g_2$ is paradigmatic.  
 Killing himself (its discoverer)  knew its structure constants, but even he found this information insufficient. 
 The first descriptions are from 1893, when Engel and Cartan independently proved that $G_2$ is the stability group of a certain Pfaff system.
A very elegant   description of $G_2$ as the group preserving a generic 3-form is due to Engel in 1900
(see   \cite{AgriG2} and references therein for more details on this story). This is the starting point for geometrical applications. 
Cartan proved in 1914 that $\f g_2$ is the derivation algebra of the octonion algebra.
The survey article    \cite{NotesG2}    gather many of the relevant results of the algebra $\f g_2$, including aspects as its relation with spinors.
Finding suitable   models of the complex Lie algebra $\f g_2$ and of its real forms makes it possible for mathematicians and physicists to use them. 
In the case of the compact algebra, 
we don't even have a root decomposition, 
so it is even more important to describe it in an accessible way.
Easy models of both the split and the compact real forms of $\f g_2$ as a sum of a classical Lie algebra with a convenient module
 are provided in \cite{modelos}; where the main tool used is multilinear algebra.
 Also a   lovely construction of the   compact real form is provided in \cite{Wilson}, not relying on the octonions.
But each model is adapted to a particular use or to a concrete symmetry. 
In this article our objective is to exploit the $\bb Z_2^3$-symmetry of $\f g_2$.
The approach is novel in that it is   based neither on its subalgebras (as \cite{modelos}) nor on octonions, but in a group.
 Of course, the octonion algebra is hidden in the construction,  but
it does not appear explicitly, so that  any reader with no prior knowledge of exceptional algebras could make use of it.
A particularly relevant fact for physicists is that the construction is not only easy, but also provides a basis with outstanding properties.

  We will describe the compact Lie algebra $\f{g}^c_2$ as a \emph{twisted ring group. }
      Several authors have remarked that the graded-division algebras with homogeneous components
of dimension 1 are nothing else but a twisted group algebras $\bb F^\sigma[G]$, for $\bb F$ a field and $G$ an abelian group. Recall that 
 the twisted group algebra consists of endowing $\bb F[G]=\{\sum\alpha_gg:g\in G,\alpha_g\in \bb F\}$ ($G$ finite) with the product $g\cdot h=\sigma(g,h)(g+h)$,
 for some map  $\sigma\colon G\times G\to \bb F$.
 With this philosophy, the octonion algebra has been described as twisted group algebra  $\bb F[\bb{Z}_2^3]$ in \cite{Helena},
 Clifford algebras as  $\bb F[\bb{Z}_2^n]$ in \cite{twisted}, and the Albert algebra as twisted group algebra  $\bb F[\bb{Z}_3^3]$ in \cite{WeylOyA};
 where $\bb F$ can be taken as $\bb R$ and $\bb C$  in the three cases.   
 This inspires our construction, but there is no grading on $\f{g}_2$ 
 which has all homogeneous components of dimension 1.

 The gradings on the complex Lie algebra $\f g_2$ were described in \cite{GradsG2,gradsg22}. This algebra has 25 gradings up to equivalence, 24 of them are toral, that is, coarsenings of the root decomposition, and only one of them is non-toral: a beautiful $\bb{Z}_2^3$-grading. Its neutral homogeneous component is zero and all the other 7 homogeneous components are two-dimensional. Moreover, the Weyl group of this grading is the whole general linear group $\mathrm{GL}(\bb{Z}_2,3)$; this means that all the homogeneous components of the grading are undistinguishable, since they are interchangeable by means of an inner automorphism. In fact, the homogeneous components  are all Cartan subalgebras, what makes that having this decomposition in a precise way is a highly desirable objective, since every basis consisting of homogeneous elements immediately would satisfy that all of its elements are semisimple.
 The two real Lie algebras of type $G_2$ are  also $\bb{Z}_2^3$-graded \cite{reales}.  
 To be more precise, the octonion division algebra is $\bb{Z}_2^3$-graded (it can be constructed by applying three times the Cayley-Dickson doubling process to the real field, see \cite{gradsO}), so that it induces a $\bb{Z}_2^3$-grading on its derivation algebra, which is the compact Lie algebra of type $G_2$; and the same is true for the split octonion algebra and its derivation algebra, which is the split Lie algebra of type $G_2$. 
 We refer to them as $\f{g}^c_{2}$ and $\f{g}^*_{2}$, respectively.

 All this suggests that any choice of basis of any of these two algebras $\la=\oplus_{g\in \bb{Z}_2^3}\la_g$ consisting of homogeneous elements for the grading would give   something close to a description as twisted ring group for the two-dimensional ring $R=\bb R\oplus\bb R$, once we remove the neutral element in the group, 
 $R[\bb{Z}_2^3\setminus\{e\}]$, for a suitable $\sigma\colon (\bb{Z}_2^3\setminus\{e\})\times (\bb{Z}_2^3\setminus\{e\})\to\mathrm{Bil}(R\times R,R)$. 
In fact, the procedure would consist of identifying the chosen basis $\{X_g,Y_g\}$ of the homogeneous component $\la_g$ with $\{(1,0)g,(0,1)g\}$. 
But this naive  idea does not mean that there is a reasonable way of doing it, at least not a priori. (This will be the procedure we will follow, for the basis defined in Eq.~\eqref{eq_EyF}.)
Material about ring groups are found for instance in \cite{Passman}, and about twisted group algebras in \cite{twisted}, but as far as we know, there is no material about \emph{twisted ring groups}.
At the moment,  we are no interested in the theoretical possibilities of this peculiar notion, but in finding convenient descriptions of the exceptional Lie algebras, in general, and of those of type $G_2$, in particular. 
We are happy to announce that we have achieved such descriptions in Theorem~\ref{main} and Corollary~\ref{main2}.

 Denote the elements in $G=\bb{Z}_2^3$ as
\begin{equation}\label{eq_indices}
 \begin{array}{cccc}
g_0 := (\bar0, \bar0, \bar0), & \quad g_1 :=(\bar1, \bar0, \bar0), & \quad g_2 := (\bar0, \bar1, \bar0), & \quad g_3 :=(\bar0, \bar0, \bar1),
\\
g_4 := (\bar1, \bar1, \bar0), & \quad g_5:= (\bar0, \bar1, \bar1), & \quad g_6:= (\bar1, \bar1, \bar1), & \quad g_7 := (\bar1, \bar0, \bar1),
\end{array}
\end{equation}
so that $g_i+g_{i+1}=g_{i+3}$, where the indices are summed modulo 7.

 \begin{theorem}\label{main}
 Let $P=\frac12\tiny\begin{pmatrix} 1&1 \\ 3&-1 \end{pmatrix}$ and $\theta= \tiny\begin{pmatrix} -1&0 \\ 0& 1\end{pmatrix}.$
   Take  the ring $R=\bb{R}\oplus\bb{R}$, with the component-wise product $(a,b)*(a',b')=(aa',bb')$.
Let the group  $G=\bb{Z}_2^3$ and denote by $G^\times=G\setminus\{(\bar0,\bar0,\bar0)\}=\{g_i:i=1\dots7\}$.
Consider the twisted ring group 
$$
 R^\sigma[G^\times]=\left\{\sum_{i=1}^7r_ig_i:r_i\in R\right\},
$$
with the product defined by,
\begin{enumerate}
\item $[rg_i,r' g_i]=0 $,
\item $[rg_i,r' g_j]=\sigma_{ij}(r,r')(g_i+g_j) $ if $i\ne j$, where $\sigma_{ij}(r,r')=-\sigma_{ji}(r',r)$ and 
$$
\begin{array}{l}
\sigma_{i,i+1}(r,r') =(r*r'P)\q\theta, \\
\sigma_{i,i+2}(r,r') = (rP*r'\q)\theta,  \\
\sigma_{i,i+4}(r,r') =  (r\q *r')\q,  
\end{array}
$$ 
 if $r,r'\in R$, $i=1\dots7$.
\end{enumerate}
Then the defined algebra is the compact Lie algebra $\f{g}^c_{2}$.
 \end{theorem}

\begin{coro}\label{masinfo}
Furthermore, $\{(4,0)g_i,(0,4)g_i:i=1\dots 7\}$ is an orthogonal basis   of semisimple elements of $\f{g}^c_{2}$ with  
 integer structure constants. More precisely, the complete table of products is given by
 $$\begin{array}{l|c|c|}
 &\ad((1,0)g_i)&\ad((0,1)g_i) \\\hline
(a,b)g_{i}&0&0\\
(a,b)g_{i+1}&\frac14 \left( a+3 b,-a-3 b\right) g_{i+3}& \frac14 \left( 3 b-3 a,b-a \right) g_{i+3}\\
(a,b)g_{i+2}& \frac14\left( a-3 b, -a-b\right) g_{i+6}&  \frac14\left( 3 a-9 b,a+b \right) g_{i+6}\\
 (a,b)g_{i+3}& \frac14 \left( 3 b-a, 3 b-a\right) g_{i+1}& \frac14 \left( 3 a+3 b,-a-b\right) g_{i+1}\\
(a,b)g_{i+4}& \frac14 \left(  a-3 b,\ a+b \right) g_{i+5}&  \frac14\left(  -3  a-3b,b-3 a\right) g_{i+5}\\
(a,b)g_{i+5}& \frac14\left(  -a-3 b, a-b \right) g_{i+4}& \frac14 \left( 3 a+9 b,a-b\right) g_{i+4}\\
(a,b)g_{i+6}& \frac14 \left(  3 b-a,\ a+b \right) g_{i+2}& \frac14 \left(  -3a-3b,3 a-b\right) g_{i+2} \\\hline
\end{array}
 $$
Also, $\{\frac12(1,0)g_i,\frac1{2\sqrt{3}}(0,1)g_i:i=1\dots 7\}$ gives an orthonormal basis (norm $-1$) of semisimple elements with  
 structure constants in $\bb Q[\sqrt 3]$.
\end{coro}

 Our main aim is to prove Theorem~\ref{main}. 
 In addition to how beautiful this model looks and how easy it is to use, there is another reason
  to develop it. There exist a fine 
 $\bb Z_3^3$-grading on $\f{f}_4$ and  a  fine $\bb Z_5^3$-grading on $\f{e}_8$ with similar properties, namely,
 all their homogeneous components are abelian and two-dimensional (\cite{EK13}). 
 The challenge, then, is to find a series of  models of the simple complex Lie algebras $\f{g}_2$, $\f{f}_4$ and $\f{e}_8$ as twisted ring groups
 $(\bb{C}\oplus\bb{C})^\sigma[(\bb Z_p^3)^\times]$ for $p=2,3,5,$ respectively, and for convenient $\sigma$'s. 
 Obviously, we should have the first of these models to get a clue about the others.
 
 As a bonus, we will be able to see $\f{g}^*_{2}$  as a twisted ring group too, according to the following description.

 \begin{coro}\label{main2}
For $R=\bb{R}\oplus\bb{R}$ and  $G=\bb{Z}_2^3$, consider the twisted ring group 
$
 R^\sigma[G^\times] =\left\{\sum_{i=1}^7r_ig_i:r_i\in R\right\}$
with the product defined by, 
\begin{enumerate}
\item $[rg_i,r' g_i]=0, $
\item $[rg_i,r' g_j]=\sigma_{ij}(r,r')(g_i+g_j) $ if $i\ne j$, where $\sigma_{ij}(r,r')=-\sigma_{ji}(r',r)$ and
$$
 \sigma (r,r')=\left\{
 \begin{array}{ll}
(r*r'P)\q\theta  \quad
 &\textrm{for }\sigma=\sigma_{12},\sigma_{13},\sigma_{34},\sigma_{45},\sigma_{71},\\
-(r*r'P)\q\theta \quad
 &\textrm{for }\sigma=\sigma_{56},\sigma_{67},\\
 (rP*r'\q)\theta \quad
    &\textrm{for }\sigma=\sigma_{13},\sigma_{24},\sigma_{46},\sigma_{61},\sigma_{72},\\
  -(rP*r'\q)\theta  \quad
    &\textrm{for }\sigma=\sigma_{35},\sigma_{57} ,\\
 (r\q *r')\q \quad
  &\textrm{for }\sigma=\sigma_{15},\sigma_{26},\sigma_{41},\sigma_{52},\sigma_{74},\\
  -(r\q *r')\q \quad
  &\textrm{for }\sigma=\sigma_{37},\sigma_{63} ,
\end{array}\right.
$$
\end{enumerate}
if $r,r'\in R$. Then the   algebra so defined is the split Lie algebra $\f{g}^*_{2}$.
\end{coro}
 
Descriptions free of  matrices are provided in Eq.~\eqref{concreto}, to be used if needed.
  Again, $\{\frac12(1,0)g_i,\frac1{2\sqrt{3}}(0,1)g_i:i=1\dots 7\}$ provides an orthonormal basis (norm $\pm1$) of semisimple elements with  
 structure constants in $\bb Q[\sqrt 3]$;
 and 
 $\{(4,0)g_i,(0,4)g_i:i=1\dots 7\}$ is an orthogonal basis of $\f{g}^*_{2}$  of semisimple elements with  
 integer structure constants. In the split algebra, this basis is  less novel, since the existence of the Chevalley basis where all structure constants are integers
is well known. 
 Anyway, in such Chevalley basis $\{h_1,h_2,e_{\alpha},e_{-\alpha}: \alpha\in\Phi^+\}$, for  $\Phi$ the root system relative to the Cartan
 subalgebra $\langle h_1,h_2\rangle$, the sign in $[e_{{\beta }},e_{{\alpha }}]=\pm (p+1)e_{{\beta +\alpha }}$ is not always clear  ($p$
 is the greatest positive integer such that $\alpha -p\beta$ is a root); so that having a concrete realization is still convenient.


 \section{Proofs and tools}\label{compacto}

 Consider the  octonion division algebra $\bb O$, that is, the real algebra with basis $\{e_i:0\le i\le7\}$ where $e_0=1$
 is the unit element, $e_i^2=-1$ and
 $$e_ie_{i+1} = e_{i+3}=-e_{i+1}e_i$$ for all $i$, summing modulo 7.
Thus $\bb O=\oplus_{g\in G} \bb O_g$  is a $G$-grading with $\bb O_{g_i}=\bb Re_i$, 
for the labelling of the elements in $G=\bb Z_2^3$   as in Eq.~\eqref{eq_indices}. For
 $i\ne  j \in I:=\{1,\dots,7\} $, we let $i \ast j$ to be the element in $I$ such that $g_i + g_j = g_{i \ast j}$. 
 This notation and others throughout the paper have been taken from   \cite[\S4.2]{arxiv}.
 Thus,
 $\{i,j,i*j\}$ is a line in the set 
 $$
 \begin{array}{ll}
 {\bf L} &= \{\{i,i+1,i+3\}:i\in I\}\\
 &=\{ \{ 1,2,4   \},\{ 2,3,5   \},\{ 3,4,6   \},\{  4,5,7  \},\{  5,6,1  \},\{  6,7,2  \},\{  7,1,3  \}   \}.
 \end{array}
 $$
For any $\ell =\{i,j,i*j\}\in {\bf L}$, consider $\mathcal Q_\ell=\langle 1,e_i,e_j,e_{i\ast j}\rangle$ the related quaternion subalgebra of $\bb O$.
For any $k\notin\ell$, note that $e_k\in   \mathcal Q_{\ell}^\perp$, so that $\bb O= \mathcal Q_{\ell}\oplus  \mathcal Q_{\ell}e_k$.

Let $\la=\Der(\bb O)$ be the derivation algebra, which is  the compact simple Lie algebra of type $G_2$. This Lie algebra is
$G=\bb Z_2^3$-graded too, for
$$
\la=\oplus_{g\in G} \la_g, \quad \la_g=\{d\in\Der(\bb O):d(\bb O_{h})\subset \bb O_{g+h}\ \forall h\in G\}.
$$
This is the grading on $\f g^c_2$ mentioned in Introduction, with seven homogeneous components $\la_{g_i}\equiv \la_i$ of dimension 2, all of them abelian.
We are interested in choosing a good basis of $\la_i$ for any $i$, in fact, several bases of the same $\la_i$.

\begin{definition} 
Take $\ell \in {\bf L}$ with $i\in\ell$, and $k \notin \ell$. %
Consider the derivations of $\bb O$ given by,  
\begin{equation}\label{eq_EyF}
\begin{array}{lll}
E^{\ell,k}_{i }: &q\mapsto 0, & qe_{k}\mapsto \frac12 (e_{i}q)e_{k},   \\
F^{\ell,k}_{i }: &q\mapsto \frac12[ e_{i},q],\qquad &qe_{k}\mapsto-\frac12 (qe_{i})e_{k}, 
\end{array}
\end{equation}
for any $q\in\mathcal Q_\ell$. These derivations are taken from  \cite{LY}. 
\end{definition}

 \begin{notation}
 For $i\ne  j \in I $, we will say that $i\prec j$ if $e_ie_j=e_{ i\ast j}$. Note that for any $i\ne j$, either $i\prec j$ or $j\prec i$.
 To be more precise, $i\prec i+k$ for $k=1,2,4$ and $i+k\prec i$ for $k=3,5,6$.
 \end{notation}

\begin{lemma}\label{lem} 
The following assertions hold, for each  $\ell \in {\bf L}$ with $i,j\in\ell$, $i\ne j$, and $k \notin \ell$.
\begin{enumerate}

\item[\rm (i)] 
 $B^{\ell,k}_{i }:=\{E^{\ell,k}_{i },F^{\ell,k}_{i }\}$ is a basis of the abelian homogeneous component
$\la_i$.

\item[\rm (ii)]  
$E^{\ell,k}_{i}=E^{\ell,i\ast k}_{i}=-E^{\ell,j \ast k}_{i}=-E^{\ell,i\ast j\ast k}_{i}$; while $F^{\ell}_{i }:=F^{\ell,k}_{i}$ does not depend on the choice of 
$k\notin\ell$.

\item[\rm (iii)] 
If $i\prec j$,   then
$[E^{\ell,k}_{i}, E^{\ell,k}_{j}] = E^{\ell,k}_{i\ast j}$.

\item[\rm (iv)] If $i\prec j$,    then
$[F^\ell_{i}, F^\ell_{j}] = F^\ell_{i\ast j}$.
\end{enumerate}
 \end{lemma}
In particular, we have got six basis of $\la_i$, for any $i\in I$.

 \begin{proof}
 Recall that $\bb O$ is alternative, $\mathcal Q_{\ell}$ is associative and, for any $s\notin\{i,j,i\ast j\}$, we have 
 \begin{equation}\label{antico}
 e_i(e_je_s)=-(e_ie_j)e_s.
 \end{equation}

(i) Clearly $E^{\ell,k}_{i }$ and $F^{\ell}_{i }$ are homogeneous derivations which send $\bb O_{h}$ into $\bb O_{g_i+h}$,
so that they belong to $\la_i$. The linear independence is clear from the following table of actions,  
\begin{equation}\label{tabla1}
\begin{array}{c|ccccccc|}
 &e_i&e_j&e_ie_j&e_k&e_{ i}e_{ k}&e_{ j}e_{ k}&(e_ie_j)e_{ k}\\\hline
2E^{\ell,k}_{i}&0&0&0&e_ie_k&-e_k&(e_ie_j)e_{ k}&-e_{ j}e_{ k}\\
2F^{\ell,k}_{i}&0&2e_ie_j&-2e_j&-e_ie_k&e_k&(e_ie_j)e_{ k}&-e_{ j}e_{ k}\\\hline
\end{array}
\end{equation}
and the result follows since $\dim\la_i=2$.

(ii) The derivations of $\bb O$ are determined by the image of $e_i,e_j,e_k$, since these three elements generate the whole octonion algebra.
As all the maps $2E^{\ell,s}_{i}$ send $\mathcal Q_{\ell}$ to $0$, we check their action on $e_k$:
$$
\begin{array}{l}
2E^{\ell,k}_{i}:e_k\mapsto e_ie_k; \quad\\
2E^{\ell,i*k}_{i}:e_k=-e_i(e_ie_k)\mapsto -e_i^2(e_ie_k)=e_ie_k;\\
2E^{\ell,j*k}_{i}:e_k=-e_j(e_je_k)\mapsto-(e_ie_j)(e_je_k)\stackrel{\eqref{antico}}=((e_ie_j)e_j)e_k=(e_ie_j^2)e_k=-e_ie_k;
\end{array}
$$
and exactly the same argument proves that $2E^{\ell,i*j*k}_{i}(e_k)=-e_ie_k$. 
Analogously, $2F^{\ell,s}_{i}\vert_{\mathcal Q_{\ell} }=\ad e_i$ for all $s\notin\ell$; 
so it just needs to be checked that $2F^{\ell,s}_{i}(e_k)=-e_ie_k$    for all $s\notin\ell$.
This is easy: for instance, 
$$
2F^{\ell,j*k}_{i}(e_k)=2F^{\ell,j*k}_{i}( -e_j(e_je_k))=+(e_je_i)(e_je_k)=-(e_ie_j)(e_je_k)=-e_ie_k.
 $$

(iii)  We now compute $E^{\ell,k}_{i }E^{\ell,k}_{j }$: it sends $q\mapsto 0$ and $qe_k\mapsto \frac14 (e_ie_jq)e_k$, for $q\in\mathcal Q_{\ell}$.
Thus $[E^{\ell,k}_{i },E^{\ell,k}_{j }]$  sends $q\mapsto 0$ and $qe_k\mapsto \frac14 ([e_i,e_j]q)e_k$.
The result follows from $[e_i,e_j]=2e_{i*j}$, since $i\prec j$.

(iv) Since $F^{\ell}_{i }\vert_{\mathcal Q_{\ell}}=\frac12\ad e_i$; then  
 $[F^{\ell}_{i },F^{\ell}_{j }]\vert_{\mathcal Q_{\ell}}=\frac14[\ad e_i,\ad e_j]=\frac14\ad[e_i,e_j]=\frac12\ad e_{i*j}$. And the action on the orthogonal subspace
 is
 $F^{\ell}_{i }F^{\ell}_{j }(qe_k)=\frac14(qe_je_i)e_k$, hence 
  $[F^{\ell}_{i },F^{\ell}_{j }](qe_k)=\frac14(q[e_j,e_i])e_k=-\frac12(qe_{i*j})e_k$.
\end{proof}

Throughout this manuscript, we denote 
 $
 P=\frac12\tiny\begin{pmatrix} 1&1 \\ 3&-1 \end{pmatrix}
 $
 and  
 $
 \theta= \tiny\begin{pmatrix} -1&0 \\ 0& 1\end{pmatrix}.
 $

 \begin{lemma}\label{cbio}
 If $\ell=\{i,j,i*j\}$, $\ell'=\{i,k,i*k\}$ are two different lines in $\bf L$, then \smallskip
\begin{itemize}
\item[(a)]
$
\ba{i}{\ell}{k}
=\frac12\begin{pmatrix}1&1\\3&-1    \end{pmatrix} \begin{pmatrix}E_i^{\ell',j}\\F_i^{\ell'}\end{pmatrix}
=P\ba{i}{\ell'}{j},
$\smallskip

\item[(b)]
$
 \ba{i}{\ell}{j*k}
=\frac12\begin{pmatrix}-1&-1\\3&-1    \end{pmatrix} \begin{pmatrix}E_i^{\ell',j}\\F_i^{\ell'}\end{pmatrix}
=\q\ba{i}{\ell'}{j}.
$\smallskip
\end{itemize}
 \end{lemma}

 \begin{proof}
First we check that $2E_i^{\ell',j}= {E_i^{\ell,k}+F_i^{\ell}}$ and
 $2F_i^{\ell'}= {3E_i^{\ell,k}-F_i^{\ell}}$. This is trivial after   comparing the table of actions in \eqref{tabla1} with
 \begin{equation*}\label{tabla2}
\begin{array}{c|ccccccc|}
 &e_i&e_j&e_ie_j&e_k&e_{ i}e_{ k}&e_{ j}e_{ k}&(e_ie_j)e_{ k}\\\hline
2E^{\ell',j}_{i}&0&e_ie_j&-e_j&0&0&(e_ie_j)e_{ k}&-e_{ j}e_{ k}\\
2F^{\ell'}_{i}&0&-e_ie_j&e_j&2e_ie_k&-2e_k&(e_ie_j)e_{ k}&-e_{ j}e_{ k}\\\hline
\end{array}
\end{equation*}
This second table is computed easily with the help of Eq.~\ref{antico}. Most of the (very easy) computations are left to the reader.
For instance,
$$
2E^{\ell',j}_{i}:e_je_k=-e_ke_j\mapsto-(e_ie_k)e_j=e_i(e_ke_j)=-e_i(e_je_k)=(e_ie_j)e_k.
$$
 This proves (a). 
 Now (b) is an immediate consequence  of   the fact that $E_i^{\ell,j*k}=-E_i^{\ell,k}$ for $\ell=\{i,j,i*j\}$,
so that the first row of $P$ is affected   by a sign change.
 \end{proof}

  \begin{notation}
If $ i\in \ell \in \bf L$,   and $k\notin\ell$, an arbitrary element in $\la_i=\{aE_i^{\ell,k}+bF_i^{\ell}:a,b\in\bb R\}$ can be written as
 $$aE_i^{\ell,k}+bF_i^{\ell}\equiv(a,b)\begin{pmatrix}E_i^{\ell,k}\\F_i^{\ell}\end{pmatrix}.$$
 Now, for $ j\in \ell $, $i\prec j$, the fact $\la_i$ abelian jointly with Lemma~\ref{lem} (iii) and (iv) permit us to claim 
\begin{equation}\label{mismalinea}
 \left[(a,b)\begin{pmatrix}E_{i}^{\ell,k}\\F_{i}^{\ell}\end{pmatrix},
 (a',b')\begin{pmatrix}E_{j}^{\ell,k}\\F_{j}^{\ell}\end{pmatrix}\right]=
 (aa',bb')\begin{pmatrix}E_{i*j}^{\ell,k}\\F_{i*j}^{\ell}\end{pmatrix}.
\end{equation}
This will be very helpful for working \emph{in the same line}, while Lemma~\ref{cbio}   
 will be relevant for handling bases related to different lines.
  \end{notation}

  We are now in a position to prove Theorem~\ref{main}.
  
  \begin{proof}
  For any $i\in I$, let us denote by $\ell_i:= \{i,i+1,i+3\}$. Consider the vector space isomorphism
   $$
   \begin{array}{rcl}
   \varphi\colon R^\sigma[G^\times]&\longrightarrow& \la=\Der(\bb O)=\f{g}^c_{2}\vspace{2pt}\\
    (a,b)g_i&\mapsto & aE_i^{\ell_i,\,i+2}+bF_i^{\ell_i}.
   \end{array}
   $$
  Let us prove that $\varphi$ defines a Lie algebra isomorphism. We  have only to check
 \begin{equation}\label{laecuacionacomprobar}
  [\varphi(rg_j),\varphi(r'g_k)]=\varphi( \sigma_{jk}(r,r')g_{j*k}),
 \end{equation}
 for all $j\prec k$, and all $r,r'\in R$;
 since,
 if $k\prec j$,   we would have
$
[\varphi(rg_j),\varphi(r'g_k)]=-
[\varphi(r'g_k),\varphi(rg_j)]=
\varphi( -\sigma_{kj}(r',r)g_{j*k})=\varphi( \sigma_{jk}(r,r')g_{j*k}).
$
 Thus we need to consider three cases, $(j,k)\in\{(s,s+1),(s+1,s+3),(s+3,s)\}$ (for an arbitrary $s\in I$).
 
 \begin{itemize}
  \item[$\bullet$] First, we apply Lemma~\ref{cbio} (a) to $i=s+1$, $j= s+2$ and $k= s+3$ to get
\begin{equation}\label{repe}
 \varphi(r'g_{s+1})=r'\ba{s+1}{\{ s+1,s+2,s+4 \}}{s+3}=r'P\ba{s+1}{\{ s+1,s+3,s \}}{s+2},
\end{equation}
 so that    
 $$
 \begin{array}{ll}
 \left[\varphi(rg_{ s}),\varphi(r'g_{ s+1})\right]&=\left[r\ba{s}{\ell_s}{s+2},r'\ba{s+1}{\ell_{s+1}}{s+3}\right]
 =\left[r\ba{s}{\ell_s}{s+2},r'P\ba{s+1}{\ell_s}{s+2}\right]\vspace{3pt}\\
 &\stackrel{\eqref{mismalinea}}=r*r'P\ba{s+3}{\ell_s}{s+2}=(r*r'P)\q\ba{s+3}{\ell_{s+3}}{s};
 \end{array}
 $$
 where here we have applied Lemma~\ref{cbio} (b) to $i=s+3$, $j= s$ and $k= s+6$, so that $j*k=s+2$.
 Moreover, Lemma~\ref{lem} (ii) says
 $$\ba{s+3}{\ell_{s+3}}{s}=\theta\ba{s+3}{\ell_{s+3}}{s+5},$$ 
 so that 
 $
 \left[\varphi(rg_{ s}),\varphi(r'g_{ s+1})\right]=\varphi((r*r'P)\q\theta g_{s+3})$. 
 This implies Eq.~\ref{laecuacionacomprobar} for $(j,k)=(s,s+1)=(i,i+1)$.\smallskip

  \item[$\bullet$] Second, we apply Lemma~\ref{cbio} (b) to $i=s+3$, $j= s+4$ and $k= s$, so that $j*k=s+5$ and we get
 \begin{equation}\label{repe2}
 \varphi(r'g_{s+3})=r'\ba{s+3}{\ell_{s+3}}{s+5}=r'\q\ba{s+3}{\{ s+1,s+3,s \}}{s+4}.
\end{equation}
As checked in Eq.~\eqref{repe}, $\varphi(rg_{ s+1})=rP\ba{s+1}{\ell_s}{s+2}=rP\ba{s+1}{\ell_s}{s+4}$, so that
$$
\left[\varphi(rg_{ s+1}),\varphi(r'g_{ s+3})\right]=
\left[ rP\ba{s+1}{\ell_s}{s+4},  r'\q\ba{s+3}{\ell_s}{s+4} \right]
\stackrel{\eqref{mismalinea}}=rP*r'\q\ba{s}{\ell_s}{s+4}.
$$
Since
$$
\ba{s}{\ell_s}{s+4}=\theta \ba{s}{\ell_s}{s+2},
$$
then we conclude 
$
\left[\varphi(rg_{ s+1}),\varphi(r'g_{ s+3})\right]=\varphi((rP*r'\q)\theta g_s).
$ 
This implies Eq.~\ref{laecuacionacomprobar} for $(j,k)=(s+1,s+3)=(i,i+2)$.\smallskip

 \item[$\bullet$]  For the third case, the bulb of the work is done. By Eq.~\eqref{repe2} and Lemma~\ref{lem} (ii),
$$
 \varphi(rg_{s+3})=r\q\ba{s+3}{\ell_s}{s+4}=r\q\ba{s+3}{\ell_s}{s+6}.
 $$
 Also $\varphi(r'g_{ s})=r'\ba{s}{\ell_s}{s+2} =r'\ba{s}{\ell_s}{s+6} $, again by Lemma~\ref{lem} (ii).
 Thus,
$$
\left[\varphi(rg_{ s+3}),\varphi(r'g_{ s})\right]=
\left[r\q \ba{s+3}{\ell_s}{s+6}, r'\ba{s}{\ell_s}{s+6}  \right]=
r\q *r'\ba{s+1}{\ell_s}{s+6}.
$$
Finally, apply Lemma~\ref{cbio} (b) to $i=s+1$, $j= s+3$ and $k= s+4$, so that $j*k=s+6$,
to get
$$
\ba{s+1}{\ell_s}{s+6}=\q\ba{s+1}{\ell_{s+1}}{s+3}. 
$$
Putting all the information together,
$$
\left[\varphi(rg_{ s+3}),\varphi(r'g_{ s})\right]= (r\q *r')\q\ba{s+1}{\ell_{s+1}}{s+3}=\varphi((r\q *r')\q g_{s+1} ).
  $$
 This implies Eq.~\ref{laecuacionacomprobar} for $(j,k)=(s+3,s)=(i,i+4)$.\smallskip
 \end{itemize}
 \end{proof}
 
 Next, we  will demonstrate the assertions on the bases in Corollary~\ref{masinfo}.
 \begin{proof}
 We have proved that 
 $$
  \left[r\ba{j}{\ell_j}{j+2},r'\ba{k}{\ell_k}{k+2}\right]= \sigma_{jk}(r,r') \ba{j*k}{\ell_{j*k}}{j*k+2},
 $$
 for $(j,k)=(i,i+1),(i,i+2),(i,i+4)$, that is, for all $j\prec k$, with
\begin{equation}\label{concreto}
 \begin{array}{l}
 \sigma_{i,i+1}(r,r')=(r*r'P)\q\theta= \frac{1}{4} \left((a-3 b) a'+3 (a+b) b',(b-3 a) b'-(a+b) a'\right),\\
 \sigma_{i,i+2}(r,r')=(rP*r'\q)\theta=\frac{1}{4} \left((a+3 b) \left(a'-3 b'\right),- (a-b) \left(a'+b'\right)\right),\\
\sigma_{i,i+4}(r,r')=(r\q *r')\q=\frac{1}{4}  \left((a-3 b) a'-3 (a+b) b',(a-3 b) a'+(a+b) b'\right).
\end{array}
\end{equation}
 This gives the grid of products in Corollary~\ref{masinfo}.
 With this table, it is easy to calculate the full 
  Killing form
   $\kappa\colon\la\times\la\to\bb R$, defined by $\kappa(x,y)=\mathrm{tr}(\ad(x)\ad(y))$.
For any grading $\la=\oplus_{g\in G}\la_g$, it is a well-known fact that $\kappa(\la_g,\la_h)=0$ if $g+h\ne e$. This implies in our case that
 $\kappa(\la_i,\la_j)=0$ if $i\ne j$.  Fix $i$   an arbitrary index, and let us denote by $F_1=\ad((1,0)g_i)$ and $F_2=\ad((0,1)g_i)$. 
We can tediously check, by following the product rules in Eq.~\eqref{concreto}, that
 $F_2F_1$ is diagonalizable with $\langle (1,0)g_j,(0,1)g_j:j=i,i+1,i+3\rangle$ its $0$-eigenspace and 
 $$
\{ (3,1)g_{i+2},
(1,-1)g_{i+2},
(1,0)g_{i+4},
(0,1)g_{i+4},
(3,-1)g_{i+5},
(1,1)g_{i+5},
(1,0)g_{i+6},
(0,1)g_{i+6}\}
$$
  eigenvectors with respective eigenvalues $\{\frac{1}{4},\frac{-3}{4}, \frac{3}{4},\frac{-1}{4},\frac{-1}{4},\frac{3}{4},\frac{-3}{4},\frac{1}{4}\}$.
This implies that the trace of $F_2F_1$ vanishes, so that  $\kappa((1,0)g_i,(0,1)g_i)=0$, for arbitrary $i$, 
and $ \{(1,0)g_i,(0,1)g_i:i\in I\}$ turns out to be an orthogonal basis for $\kappa$.
Finally, we compute the norm of the elements in this basis. First, $F_1^2$ is diagonalizable with eigenvectors related to:
\begin{itemize}
\item Eigenvalue 0: $\{ (1,0)g_i,(0,1)g_i, (3,-1)g_{i+1},(3,1)g_{i+3} \}$;
\item Eigenvalue $-1$: $\{ (1,1)g_{i+1},(1,-1)g_{i+3}  \}$;
\item Eigenvalue $\frac{-1}4$: $\{  (1,0)g_j,(0,1)g_j: j=i+2,i+4,i+5,i+6\}$;
\end{itemize}
so that $\kappa((1,0)g_i,(1,0)g_i)=\mathrm{tr}(F_1^2)=-4$.
Similarly, $F_2^2$ is diagonalizable with eigenvectors related to:
\begin{itemize}
\item Eigenvalue 0: $\{ (1,0)g_i,(0,1)g_i, (1,1 )g_{i+1},(1,-1 )g_{i+3} \}$;
\item Eigenvalue $-1$: $\{ (3,1 )g_{i+1},(3,1 )g_{i+3}  \}$;
\item Eigenvalue $\frac{-1}4$: $\{  (3,1)g_{i+2},(0,1)g_{i+4},(3,-1)g_{i+5},(0,1)g_{i+6} \}$;
\item Eigenvalue $\frac{-9}4$: $\{   (1,-1)g_{i+2},(1,0)g_{i+4},(1,1)g_{i+5},(1,0)g_{i+6}\}$;
\end{itemize}
so that $\kappa((0,1)g_i,(0,1)g_i)=\mathrm{tr}(F_2^2)=-12$. The rest of the proof follows trivially.
  \end{proof}

\noindent  In fact, we have directly seen that the Killing form is negative definite, as expected, since the Lie algebra $\Der(\bb O)$ is compact.

 Finally, the model for the split algebra described in  Corollary~\ref{main2} follows easily from  the compact model.
  
  \begin{proof}
 If $\la$ is the compact Lie algebra of type $G_2$ with the model described in Theorem~\ref{main}, 
 then $\la=\la_{\bar0}\oplus\la_{\bar1}$ is a $\bb Z_2$-grading for
 $$
 \la_{\bar0}=\la_{ 1}\oplus\la_{2 }\oplus\la_{ 4};
 \qquad
 \la_{\bar1}=\la_{ 3}\oplus\la_{5 }\oplus\la_{6 }\oplus\la_{ 7}.
 $$
 Then it is well known that $\la'=\la_{\bar0}\oplus{\bf i}\la_{\bar1}$ is the split Lie algebra of type $G_2$, $\bb Z_2$-graded too. If
 $\{v_i:i=1\dots 6\}$ is a basis of $\la_{\bar0}$ and $\{v_i:i=7\dots 14\}$  is a basis of $\la_{\bar1}$, then
 $\{v_i:i=1\dots 6\}$ is a basis of $\la'_{\bar0}$ and $\{{\bf i} v_i:i=7\dots 14\}$  is a basis of $\la'_{\bar1}$. 
 Relative to this basis the structure constants are the same 
 as in $\la$, except when we multiply two odd elements, in which case they change sign.
 The grading chosen for our description in Corollary~\ref{main2} has been that one induced by the line $\ell_1=\{1,2,4\}$.
  \end{proof}

 \smallskip
 
{ \bf Acknowledgements } We would like to thank  Alberto Elduque 
for his wonderful suggestion about the indexes $i,i+1,i+3$, which   obviously has  greatly improved the symmetry of the model.

\end{document}